\documentclass[11pt]{amsart}
\usepackage{graphicx} 

\usepackage[utf8]{inputenc}
\usepackage[margin=1in,letterpaper,portrait]{geometry}
\usepackage{ytableau}
\ytableausetup{notabloids}
\ytableausetup{centertableaux}
\usepackage{latexsym,amsmath,amssymb,verbatim, tikz}
\usepackage{graphicx}
\usepackage{hyperref}

   \newtheorem{theorem}{Theorem}[section]
   \newtheorem{proposition}[theorem]{Proposition}
   
   \newtheorem{lemma}[theorem]{Lemma}
   \newtheorem{corollary}[theorem]{Corollary}
   \newtheorem{conjecture}[theorem]{Conjecture}
   \newtheorem{problem}[theorem]{Problem}
   \newtheorem{observation}[theorem]{Observation}

   \newtheorem{definition}[theorem]{Definition}
   \newtheorem{defn}[theorem]{Definition}

   \newtheorem{remark}[theorem]{Remark}

\numberwithin{equation}{section}

\newcommand{\todo}[1]{\vspace{5 mm}\par \noindent
\marginpar{\textsc{ToDo}} \framebox{\begin{minipage}[c]{0.95
\textwidth}
 #1 \end{minipage}}\vspace{5 mm}\par}

\newcommand{\RR}{{\mathbb {R}}}
\newcommand{\ZZ}{{\mathbb {Z}}}

\newcommand{\inv}{{\operatorname{inv}}}
\newcommand{\minv}{{\operatorname{minv}}}
\newcommand{\winv}{{\operatorname{winv}}}
\newcommand{\cwinv}{{\operatorname{cwinv}}}

\newcommand{\dist}{{\operatorname{dist}}}

\newcommand{\D}{{\operatorname{Diameter}}}
\newcommand{\Sort}{{\operatorname{Sort}}}

\newlength{\mysizetiny}
\newlength{\mysizesmall}
\newlength{\mysize}
\newlength{\mysizelarge}
\title[Conjugating full cycles]
{Conjugating full cycles by adjacent transpositions: \\
diameter and sorting time}

\author{Ron M.\ Adin}
\address{Department of Mathematics, Bar-Ilan University, Ramat-Gan 52900, Israel}
\email{radin@math.biu.ac.il}

\author{Eli Bagno}
\address{Jerusalem College of Technology and Michlala College, Jerusalem, Israel}
\email{bagnoe@jct.ac.il}

\author{Yuval Roichman}
\address{Department of Mathematics, Bar-Ilan University, Ramat-Gan 52900, Israel}
\email{yuvalr@math.biu.ac.il}

\date{September 3, 2026}


\begin{document}

\begin{abstract}
    We establish upper and lower bounds on the maximal number of steps needed to transform a cyclic permutation to the canonical cyclic permutation  
    using conjugation by adjacent transpositions, and on the diameter of the underlying Schreier graph.  
\end{abstract}

\maketitle

\tableofcontents 


\section{Introduction}\label{sec:introduction}


Let $C_n$ be the set of $n$-cycles in the symmetric group $S_n$. 
Let $\Phi_n := \{s_i = (i,i+1) :\, 1 \le i < n\}$ be the set of adjacent transpositions (Coxeter simple reflections) in $S_n$.

\begin{definition}\label{def:1}
    Consider the undirected graph $\Gamma_n$, with vertex set $C_n$, in which two cycles $c_1,c_2 \in C_n$ are adjacent if $c_2 = s_i c_1 s_i$ for some adjacent transposition $s_i \in \Phi_n$.     
\end{definition}


\begin{problem}\label{pb:1}
    Find the diameter of the graph $\Gamma_n$,
    \[
        D_n := \D(\Gamma_n)
        = \max_{\beta,\gamma\in C_n} \dist_{\Gamma_n}(\beta,\gamma).
    \]
Here $ \dist_{\Gamma_n}(\beta,\gamma)$ is the 
distance (length of shortest path) in $\Gamma_n$ between $\beta$ and $\gamma$. 
\end{problem}



A closely related problem is the following. 
A {\em cyclic permutation} $[\pi]$ of size $n$ is 
a bijective labeling of $n$ distinct points on a circle 
by the elements of $[n] := \{1, 2,\dots, n\}$; two labelings are equivalent if one can be obtained from the other by a cyclic shift. 
 The {\em canonical cyclic permutation} $c_n: = (1, 2, \dots , n)$ is the one in which the elements are clockwise labeled by $1,\dots,n$. 
A {\it permissible step} on $[\pi]$ is a switching of two consecutive values, $i$ and $i+1$, for some $1 \le i < n$. 


\begin{definition}\label{df:2}
     The {\em  Coxeter cyclic sorting time}, $\Sort_n$,  is the  maximal number of permissible steps needed to convert a cyclic permutation $[\pi]$ of size $n$ into the canonical cyclic permutation $c_n = (1, 2, \dots , n)$.  
\end{definition}

     By Remark~\ref{rem:iso} below, 
    \[
        \Sort_n
        := \max_{\gamma\in C_n} \dist_{\Gamma_n}(c_n, \gamma). 
    \]   

\begin{problem}\label{pb:2}



    Find the  Coxeter cyclic sorting time. 

\end{problem}

For other sorting processes on cyclic permutations see~\cite{AAR} and references therein. 


By definition, the  Coxeter cyclic sorting time is not bigger than the diameter.
Since the graph $\Gamma_n$ is not vertex transitive for $n\ge 4$ (see, e.g., Figure~\ref{fig:gamma4}), these two numbers are not necessarily equal. 
For example, for $n=5$, $\D(\Gamma_5)=5>4=\Sort_5$. 

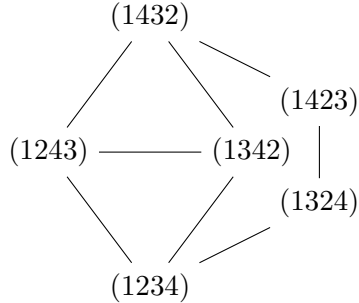
\begin{figure*}[ht]
\begin{center} 
\begin{tikzpicture}[scale=.45]
\node (2) at (0,4) {$(1432)$};
\node (1) at (-3,0) {$(1243)$};
\node (1') at (3,0) {$(1342)$};
\node (0) at (0,-4) {$(1234)$};
\node (A1) at (5,-1.5) {$(1324)$};
\node (A2) at (5,1.5) {$(1423)$};
\draw (2)--(1)--(0)--(1')--(2)--(A2)--(A1)--(0);
\draw  (1)--(1');
\end{tikzpicture}
\end{center}
    \caption{The graph $\Gamma_4$.}
    \label{fig:gamma4}
\end{figure*}


In this paper we prove
\begin{theorem}\label{thm:main1}
    For every positive integer $n$, 
    \[
         \frac{8-\pi}{16}\cdot n^2 - \frac{3n}{2} 
        \le \Sort_n  
        \le \frac{1}{3}\cdot n^2 - \frac{3n-1}{6} 
    \]
    and
    \[
         \frac{8-\pi}{16}\cdot n^2 - \frac{3n}{2} 
        \le \D(\Gamma_n)
         \le \frac{3}{8}\cdot n^2 - \frac{4n-1}{8}. 
    \]     
\end{theorem}

See Theorems~\ref{thm:u0}, ~\ref{thm:2} and~\ref{thm:1} below.
Note that $\frac{8 - \pi}{16} \sim 0.30365$.


\begin{remark}
Theorem~\ref{thm:main1} implies that the diameter and sorting time discussed in this paper, where permissible steps switch adjacent values, are different from those studied in~\cite{AAR}, where steps switch adjacent positions. 
\end{remark}

\section{Notation and preliminaries}

Let $S_n$ be the symmetric group on the letters $[n]:=\{1,\dots, n\}$. The one-line notation of a permutation $\pi\in S_n$ is denoted by square brackets $[\pi(1),\dots,\pi(n)]$, while the cycle notation by parentheses. For example $[3,5,6,1,2,4]=(1,3,6,4)(2,5)\in S_6$.

Letting $H$ be a subgroup of a finite group $G$, and $S\subseteq  G$, the {\em Schreier graph} $X(G/H,S)$ is the undirected graph whose vertex set consists of all left cosets $G/H:=\{\pi H\, :\, \pi\in G\}$; the edges are the unordered pairs $(\pi H, s\pi H)$, for $s\in S, \pi\in G$.

\section{The minv statistic}


Let $c_n:=(1,2,\dots,n)=s_1s_2\cdots s_{n-1}\in S_n$. 
Denote by $\ZZ_n$ the cyclic subgroup of $S_n$ generated by $c_n$. Recall the graph $\Gamma_n$ from Definition~\ref{def:1}.



\begin{remark}\label{rem:iso}
    Consider the map from $C_n$ to the set of left cosets $S_n/\ZZ_n$, defined as follows: a cycle $\gamma = (a_1,\dots,a_n)\in C_n$ is mapped to the coset $\bar\gamma \ZZ_n$, where $\bar\gamma=[a_1,\dots,a_n]$ . This map is a well-defined bijection and induces an isomorphism from the graph $\Gamma_n$ to the Schreier graph $X(S_n/\ZZ_n, \Phi_n)$, where edge multiplicity is ignored. 
\end{remark}


Recall the {\em inversion number} of a permutation $\pi\in S_n$,
\[
    \inv(\pi) := |\{i<j :\, \pi(i) > \pi(j)\}|.
\]

\begin{definition}\label{def:2}
    Define the {\em minv} statistic on the left cosets of $\ZZ_n$ in $S_n$ as follows:
    \[
        \minv(\pi \ZZ_n) := \min\{ \inv(\sigma) :\, \sigma \in \pi\ZZ_n)\}
        \qquad (\forall\, \pi\ZZ_n \in S_n/\ZZ_n).
    \]
\end{definition}



\begin{lemma}\label{lem:2} 
    For every pair of full cycles $\gamma_1, \gamma_2\in C_n$,
    \[
        \dist_{\Gamma_n}(\gamma_1,\gamma_2)=\min 
        \{\inv(\sigma):\ \sigma \in \bar\gamma_1 \ZZ_n \bar\gamma_2^{-1}\},  
    \]
    where $\bar\gamma_1, \bar\gamma_2$ are defined as in Remark~\ref{rem:iso}. 
\end{lemma}

\begin{proof} 
    By Remark~\ref{rem:iso}, the distance $\dist_{\Gamma_n}(\gamma_1, \gamma_2)$  
    is equal to the smallest nonnegative integer $\ell$ for which there exist simple reflections $s_{i_1},\ldots,s_{i_\ell}$ satisfying  $s_{i_\ell} \cdots s_{i_1} \bar\gamma_1\ZZ_n = \bar\gamma_2\ZZ_n$.  
    Equivalently, 
    \[
        \ell 
        = \min \{\inv(\tau_2 \tau_1^{-1}) :\, \tau_1 \in \bar\gamma_1\ZZ_n,\, \tau_2 \in \bar\gamma_2\ZZ_n\}
        = \min \{\inv(\sigma) :\, \sigma \in \bar\gamma_2 \ZZ_n \bar\gamma_1^{-1}\}.
        \qedhere
    \]
\end{proof}

\begin{corollary}\label{lem:1} 
    For every full cycle $\gamma \in C_n$
    \[
        \dist_{\Gamma_n}(c_n, \gamma) = \minv(\bar\gamma\ZZ_n),   
    \]
    where $c_n = (1,2,\dots,n) \in C_n$ and $\bar\gamma\ZZ_n \in S_n/\ZZ_n$ is defined as in Remark~\ref{rem:iso}.\\
    In particular, for every $n\ge 1$ 
    \[
        \Sort_n
        = \max_{\pi\in S_n} \minv(\pi \ZZ_n) .
    \]
\end{corollary}

We conclude that Problem~\ref{pb:2} is equivalent to the following one. 

\begin{problem}

    What is the maximal value of {\em minv} on $S_n/\ZZ_n$ ?
    

\end{problem}

\section{Upper bound on the Coxeter cyclic sorting time}\label{sec:upper sorting}



Recall the  Coxeter cyclic sorting time, $\Sort_n$, from Definition ~\ref{df:2}. The main result in this section is the following. 



\begin{theorem}\label{thm:u0}
        For every $n\ge 1$,
    \[
        \Sort_n 
        \le \frac{2n^2 - 3n + 1}{6}.
    \]
\end{theorem}


Theorem~\ref{thm:u0} is an immediate consequence of the following result. 




\begin{proposition}\label{prop:u0}
    For every $n\ge 1$, the maximum of the expected value of $\inv$ on left cosets $\pi \ZZ_n$ satisfies 
    \[
        \max_{\pi \in S_n} E\{ \inv(\sigma): \sigma \in  \pi \ZZ_n\}
        = \frac{(2n-1)(n-1)}{6}.
    \]
\end{proposition}

See Remark~\ref{rem:max} below.

\begin{proof}[Proof of Theorem~\ref{thm:u0}]
    By 
    Corollary~\ref{lem:1}, Definition~\ref{def:2} 
    and Proposition~\ref{prop:u0},
\begin{align*}
    \Sort_n 
    &= \max_{\pi\in S_n} \minv(\pi \ZZ_n)
    = \max_{\pi\in S_n} \min \{\inv(\pi c_n^i): 0\le i<n\} \\
    &\le \max_{\pi\in S_n} E\{\inv(\pi c_n^i): 0\le i<n\}
    = \frac{(2n-1)(n-1)}{6}. 
    \qedhere
\end{align*}

\end{proof}

In order to prove Proposition~\ref{prop:u0} 
we consider a weighted version of the inversion number, as well as a cyclic variant, see Definitions~\ref{def:winv} and~\ref{def:cyclic winv} below. 

\begin{defn}\label{def:winv}
    The {\em weighted inversion number} of a permutation $\pi\in S_n$ is 
    \[
        \winv(\pi)
        := \sum_{\substack{i,j \\ i<j,\, \pi(i)>\pi(j)}} (\pi(i)-\pi(j)).
    \]
\end{defn}

\begin{lemma}\label{lem:u6}
    For every $\pi \in S_n$,
    \[
        \winv(\pi) 
        = \sum_{i=1}^n i^2 -
        \sum_{i=1}^n i\cdot \pi(i) .
    \]
\end{lemma}

\begin{proof}
    By induction on the inversion number $\inv(\pi)$. 
    If $\inv(\pi) = 0$ then $\pi = id$, and the claim clearly holds. 
    Now assume that the claim holds for all $\pi' \in S_n$ with $\inv(\pi') \le k$, and let $\pi \in S_n$ have $\inv(\pi) = k+1$. 
    Since $\pi \ne id$, there exists $1 \le j<n$ such that $\pi(j) > \pi(j+1)$, so that $\pi' := \pi s_j$ satisfies
    $\inv(\pi') = k$. 
    By the induction hypothesis,
    \[
        \winv(\pi') 
        = \sum_{i=1}^n i^2 -
        \sum_{i=1}^n i\cdot \pi'(i) .
    \]    
    By Definition~\ref{def:winv}, $\winv(\pi) - \winv(\pi') = \pi(j) - \pi(j+1)$. Also,
    \[
        \left( \sum_{i=1}^n i^2 -
        \sum_{i=1}^n i\cdot \pi(i) \right)
        - \left( \sum_{i=1}^n i^2 -
        \sum_{i=1}^n i\cdot \pi'(i) \right)
        = \sum_{i=1}^n i \cdot \left( \pi'(i) - \pi(i) \right)
        = \pi(j) - \pi(j+1),
    \]
    and thus the claim holds for $\pi$ as well.
    This completes the induction step and the proof.
\end{proof}

Denote $w_0:=[n,n-1,\dots,1]\in S_n$. Recall that $w_0$ has the largest inversion number in $S_n$. 

\begin{observation}\label{obs:u2}
    For any $\pi \in S_n$,
        \[
            \winv(\pi)+\winv(\pi w_0)
            = \sum_{1 \le k < \ell \le n} (\ell - k)
            = \sum_{\ell = 2}^{n} \binom{\ell}{2} = \binom{n+1}{3}
        \]
        and
        \[
            0 = \winv(id) 
            \le \winv(\pi)
            \le \winv(w_0)
            = \binom{n+1}{3}.
        \]
\end{observation}

The following cyclic variant of $\winv$ plays a key role in the proof.

\begin{defn}\label{def:cyclic winv}
    The {\em cyclic weighted inversion number} of a permutation $\pi\in S_n$ is
    \[
        \cwinv(\pi)
        := n\cdot \inv(\pi)-2\cdot \winv(\pi)
        = \sum_{\substack{i,j \\ i<j,\, \pi(i)>\pi(j)}} (n-2\pi(i) + 2\pi(j)).
    \]
\end{defn}

The following observation justifies the adjective {\em cyclic} in Definition~\ref{def:cyclic winv}. 

\begin{observation}\label{obs:u3}
The cyclic weighted inversion number is invariant under cyclic shifts; namely, 
\[
    \cwinv(\pi c_n) = \cwinv(\pi) 
    \qquad (\forall \pi\in S_n).
\]
\end{observation}

\begin{proof}
    Denote $m:=\pi(1)$. Then $m = \pi c_n(n)$, so that
    \[
        \inv(\pi c_n)-\inv(\pi)
        = (n-m)-(m-1) = n+1-2m
    \]
    and
    \[
        \winv(\pi c_n)-\winv(\pi)
        = \sum_{m < j \le n} (j-m) -\sum_{1 \le j < m} (m-j)
        = \binom{n-m+1}{2} - \binom{m}{2}.
    \]
    Thus
    \[
        \cwinv(\pi c_n)-\cwinv(\pi)
        = n(n+1-2m)-(n-m+1)(n-m)+m(m-1)
        = 0. \qedhere
    \]
\end{proof}

\begin{observation}\label{obs:u4}
    For every $\pi\in S_n$
    \[
    \cwinv(\pi)+\cwinv(\pi w_0)= n\binom{n}{2}-2\binom{n+1}{3}=\binom{n}{3}.
    \]
\end{observation}

\begin{corollary}\label{obs:u5}
    For every $\pi\in S_n$
    \[
    0=\cwinv(id)\le \cwinv(\pi)\le \cwinv(w_0)=
    \binom{n}{3}.
    \]
\end{corollary}

\begin{proof}
    By Observation~\ref{obs:u3}, we may assume that $\pi(n)=n$.
    Let $\bar\pi=[\pi(1),\dots,\pi(n-1)]\in S_{n-1}$. By definition,
    \[
    \cwinv(\pi)=n\cdot \inv(\pi)-2\cdot \winv(\pi)=
    n\cdot \inv(\bar \pi)-2\cdot \winv(\bar \pi)=
    \cwinv(\bar \pi)+ \inv(\bar \pi).
    \]
Hence, 
\[
\max_{\pi\in S_n} \cwinv(\pi)\le \max_{\bar \pi\in S_{n-1}} \cwinv(\bar \pi)+\max_{\pi \in S_{n-1}}\inv(\pi). 
\]
By iteration, 
\[
    \max_{\pi\in S_n} \cwinv(\pi)
    \le \max_{\bar \pi\in S_{n-1}} \cwinv(\bar \pi)+\max_{\pi \in S_{n-1}}\inv(\pi)
    \le \sum_{j=1}^{n-1} \max_{\pi\in S_j} \inv(\pi)
    = \sum_{j=1}^{n-1}\binom{j}{2}
    = \binom{n}{3} 
    = \cwinv(w_0).
\]
Combining this upper bound with Observation~\ref{obs:u4}, 
\[
    \min_{\pi\in S_n} \cwinv(\pi)
    = \binom{n}{3}-\max_{\pi\in S_n}\cwinv(\pi)
    \ge 0 
    = \cwinv(id). \qedhere
\]
\end{proof}

\begin{lemma}\label{lem:u7}
    For every $\pi \in S_n$
    \[
        n\cdot  
        E\{ \inv(\sigma): \sigma \in  \pi \ZZ_n\}
        = \sum_{i=1}^n \inv(\pi c_n^i)
        = \cwinv(\pi)+ \binom{n+1}{3}.
    \]
\end{lemma}

\begin{proof}
As in the proof of Observation~\ref{obs:u3},
\[
    \inv(\pi c_n) 
    = \inv(\pi) + n + 1 - 2\pi(1). 
\]
Thus, by iteration
\[
    \inv(\pi c_n^i)
    = \inv(\pi) + \sum_{j=1}^{i} (n + 1 - 2\pi(j)).
\]
Summing over $1\le i\le n$ we obtain
\begin{align*}
    \sum_{i=1}^n \inv(\pi c_n^i)
    &= n \cdot \inv(\pi) + \sum_{i=1}^n \sum_{j=1}^{i} (n + 1 - 2\pi(j)) \\
    &= n \cdot \inv(\pi) + \sum_{j=1}^{n}(n + 1 - j)(n + 1 - 2\pi(j)) \\
    &= n \cdot \inv(\pi) + \sum_{j=1}^{n} j \cdot 2\pi(j) - (n+1) \sum_{j=1}^{n} (2\pi(j) + j) + n(n+1)^2 \\
    &= n \cdot \inv(\pi) + 2\sum_{j=1}^{n} j \cdot \pi(j) - \frac{1}{2}n(n+1)^2
\end{align*}
By Lemma~\ref{lem:u6} and Definition~\ref{def:cyclic winv}, this is equal to
\[
    n \cdot \inv(\pi) + 2\sum_{j=1}^{n} j^2 - 2 \cdot \winv(\pi) - \frac{1}{2}n(n+1)^2
    = \cwinv(\pi)+\binom{n+1}{3}. \qedhere
\]
\end{proof}

\begin{proof}[Proof of Proposition~\ref{prop:u0}]
    By Lemma~\ref{lem:u7} and Corollary~\ref{obs:u5},  
    \begin{align*}
        \max_{\pi \in S_n} E\{ \inv(\sigma): \sigma \in  \pi \ZZ_n\}
    =& \frac{1}{n}\left(\max_{\pi\in S_n} \cwinv(\pi) + \binom{n+1}{3}\right)\\
    =& \frac{1}{n}\left( \binom{n}{3} + \binom{n+1}{3} \right)
    = \frac{(2n-1)(n-1)}{6}.   \qedhere 
\end{align*}
\end{proof}





\begin{remark}\label{rem:max}
    By Corollary~\ref{obs:u5} and Lemma~\ref{lem:u7}, the maximum of the expected value of the inversion number on a $\ZZ_n$-coset is attained at $w_0 \ZZ_n$. However, surprisingly, the maximal value of the statistic $\minv$ 
    (which, by Corollary~\ref{lem:1}, is equal to $\Sort_n$)
    is not attained at this coset, since 
    $\minv(w_0 \ZZ_n) = \binom{\lceil n/2\rceil}{2} + \binom{\lfloor n/2\rfloor}{2} = \left\lfloor \frac{(n-1)^2}{4} \right\rfloor $ is strictly less than the lower bound proved in the next section (Theorem~\ref{thm:2}). 
\end{remark}

\section{Lower bound on the Coxeter cyclic sorting time}

In this section we prove the following lower bound on $\Sort_n$, and thus on $\D(\Gamma_n)$ as well.

\begin{theorem}\label{thm:2}
    For every $n\ge 1$
    \[
    \Sort_n 
    \ge \left( \frac{1}{2}-\frac{\pi}{16} \right) n^2 -\frac{3}{2} n. 
    \]
\end{theorem}


The proof of this theorem uses the following characterization.

\begin{lemma}\label{lem:lower1}
    For every $\pi\in S_n$, $\inv(\pi)=\minv(\pi\ZZ_n)$ if and only if $\pi$ is ``heavy tailed'', namely if, for all $1\le k\le n$, the sum of the last $k$ entries of $\pi$ is at least $\frac{k(n+1)}{2}$;
    equivalently if, for all $1\le k\le n$, the sum of the first $k$ entries of $\pi$ is at most $\frac{k(n+1)}{2}$:
    \[
        \sum_{j=1}^{k} \pi(j) \le \frac{k(n+1)}{2}
        \qquad (\forall k);
   \]
\end{lemma}

\begin{proof} 
    For every $1\le k\le n$, $\inv(\pi)\le \inv(\pi c_n^k)$ if and only if 
    \begin{align*}
        0 \le \inv(\pi c_n^k) - \inv(\pi) 
        =& \sum_{j=1}^{k} \left( \inv(\pi c_n^j) - \inv(\pi c_n^{j-1}) \right) \\
        =& \sum_{j=1}^k \left( (n - \pi c_n^{j-1}(1)) - (\pi c_n^{j-1}(1) - 1) \right) \\
        =& \sum_{j=1}^k (n+1-2\pi(j))
        = k(n+1) - 2\sum_{j=1}^k \pi(j). 
        \qedhere
    \end{align*}
\end{proof}

\begin{proof}[Proof of Theorem~\ref{thm:2}]
    By Corollary~\ref{lem:1}, if $\pi \in S_n$ satisfied the conditions of Lemma~\ref{lem:lower1} then
    \[
        \Sort_n
        \ge \minv(\pi \ZZ_n) 
        = \inv(\pi).
    \]
To get a good lower bound, 
we shall construct such a permutation $\pi\in S_n$ with a large  $\inv(\pi)$.


Recall the permutation $w_0 = [n,\ldots,1]$ and consider, as a first approximation, its cyclic shift $w_0 c_n^{m} =[n-m,\ldots,1,n,\ldots,n-m+1]$, where $m:=\lfloor n/2\rfloor$.  
Noting that $w_0 c_n^{m}$ satisfies the conditions of Lemma~\ref{lem:lower1}, one concludes that
\[
    \minv(w_0\ZZ_n)
    = \inv(w_0 c_n^{m})
    = \binom{n-m}{2}+\binom{m}{2}=\binom{\lceil n/2\rceil}{2}+\binom{\lfloor n/2 \rfloor}{2}.
\]
Now, take $w_0 c_n^{m}$ and push the letter $n$ to the left 
as much as possible, say to position $k_1 + 1$, so that the conditions of Lemma~\ref{lem:lower1} are still satisfied.
Then push the letter $n-1$ to the left as much as possible, say to position $k_2 + 2$, and so on. We obtain the permutation
\begin{align}\label{eq:defpi}
    \pi_0
    &:= \left[
    n-m,n-m-1,\dots, n-m-(k_1-1),n,n-m-k_1,\dots, n-m-(k_2-1),n-1, \right. \\
    & \quad\; \left. n-m-k_2,\dots, 
    n-m-(k_t-1),n+1-t,n-m-k_t,\dots \right] \in S_n, \nonumber
\end{align}
where $0 \le k_1 \le \ldots \le k_{m} \le n-m$ have the minimal possible values such that $\pi_0$ still satisfies the conditions of Lemma~\ref{lem:lower1}. In the sequel we shall obtain explicit expressions for these minimal values, for each $n$; 
for example, for $n=12$, $\pi_0 =[6,5,4,3,12,2,11,1,10,9,8,7]$. 

Let us compute $\inv(\pi_0)$, which will serve as a lower bound for $\Sort_n$.
Considering separately the small values $1 \le i \le m$ and the large values $m+1 \le n+1-t \le n$,
\[
    \minv(\pi_0\ZZ_n)
    = \inv(\pi_0)
    = \sum_{i=1}^{m} (i-1) + \sum_{t=1}^{m}(n-t-k_t)
    = \binom{n}{2} - \sum_{t=1}^{m} k_t.
\]

\medskip

First, assume 
that $n = 2m$ is even.
By the conditions of Lemma~\ref{lem:lower1}, for every $1 \le t \le m$,
\[
    \frac{(k_t+t)(n+1)}{2}
    \ge \sum_{j=1}^{k_t+t} \pi_0(j)
    = \sum_{i=0}^{k_t-1} \left( m-i \right) + \sum_{j=0}^{t-1} \left( n-j \right)
    = \frac{k_t (n+1-k_t)}{2} + \frac{t (2n+1-t)}{2}.
\]
Simplifying, this yields 
\[
    k_t^2 \ge t(n-t)
    \qquad (\forall\, 1\le t \le m).
\]
If these inequalities are satisfied, then the inequalities of Lemma~\ref{lem:lower1} for indices other than $k_1+1, \cdots, k_m+m$ obviously follow.
Minimizing the parameters $k_t$, for all $1\le t\le m$, we obtain approximately
\[
    \sum_{t=1}^{m} k_t
    \sim \sum_{t=1}^{n/2} \left( t(n-t) \right)^{1/2}
    \sim \int_0^{n/2} \left( x(n-x) \right)^{1/2} dx.
\]
Letting $x = \frac{n}{2} (1 - \cos\theta)$, 
we obtain 
\begin{align*}
    \int_0^{n/2}\left (x(n-x)\right)^{1/2} dx
    = \frac{n^2}{4} \int_0^{\pi/2} (1 - \cos^2 \theta)^{1/2} \sin\theta d\theta 
    = \frac{\pi}{16}n^2.
\end{align*}
This yields 
\[
    \Sort_n
    = \max_{\pi \in S_n} \minv(\pi \ZZ_n)
    \ge \minv(\pi_0 \ZZ_n)
    = \binom{n}{2} - \sum_{t=1}^m k_t 
    \sim \binom{n}{2} - \frac{\pi}{16} n^2.
\]

For a more precise lower bound, let us bound the error in this approximation. Since
\[
    k_t 
    = \left\lceil (t(n-t))^{1/2} \right\rceil
    \qquad (1 \le t \le m),
\]
we conclude that
\begin{equation}\label{eq:even_error1}
    0 
    \le \sum_{t=1}^{n/2} k_t - \sum_{t=1}^{n/2} (t(n-t))^{1/2}
    \le n/2.    
\end{equation}
On the other hand, $f(x) := \left(x(n-x)\right)^{1/2}$ is a monotone increasing function for $0\le x\le n/2$, and therefore 
\[
    f(t-1) \cdot 1 
    \le \int_{t-1}^{t} f(x)dx
    \le f(t)\cdot 1
    \qquad (1 \le t \le n/2).
\]
Hence 
\[
    0
    \le \sum_{t=1}^{n/2} f(t) - \int_0^{n/2} f(x)dx
    \le \sum_{t=1}^{n/2} f(t) - \sum_{t=1}^{n/2} f(t-1)
    = f \left( \frac{n}{2} \right)-f(0)
    = \frac{n}{2}
\]
or, explicitly,
\begin{equation}\label{eq:even_error2}
    0 
    \le \sum_{t=1}^{n/2} (t(n-t))^{1/2} - \frac{\pi}{16} n^2
    \le \frac{n}{2}.    
\end{equation}
Adding inequalities~\eqref{eq:even_error1} and~\eqref{eq:even_error2} yields
\[
    0 
    \le \sum_{t=1}^{n/2} k_t  - \frac{\pi}{16} n^2
    \le n.    
\]
Thus, for even $n$:
\[
    \Sort_n\ge \inv(\pi_0) 
    = \binom{n}{2} - \sum_{t=1}^{n/2} k_t  
    \ge \left( \frac{1}{2} - \frac{\pi}{16} \right) n^2 - \frac{3}{2} n.
\]

\medskip

Now assume that $n = 2m+1$ is odd. 
The computations are similar to those for even $n$, 
but a little more complicated.
%
%
By the conditions of Lemma~\ref{lem:lower1}, for every 
$1 \le t \le m$
\[
    \frac{(k_t+t)(n+1)}{2}
    \ge \sum_{j=1}^{k_t+t} \pi_0(j)
    = \sum_{i=0}^{k_t-1} \left( m+1-i \right) + \sum_{j=0}^{t-1} \left( n-j \right)
    = \frac{k_t (n+2-k_t)}{2} + \frac{t (2n+1-t)}{2}.
\]
Simplifying, this yields 
\[
    k_t^2 - k_t \ge t(n-t)
    \qquad (\forall\, 1\le t \le m)
\]
or, equivalently,
\[
    k_t \ge (t(n-t) +1/4)^{1/2} + 1/2
    \qquad (\forall\, 1\le t \le m).
\]
Again,
\[
    \minv(\pi_0\ZZ_n)
    = \inv(\pi_0)
    = \sum_{i=1}^{m+1} (i-1) + \sum_{t=1}^{m}(n-t-k_t)
    = \binom{n}{2} - \sum_{t=1}^{m} k_t.
\]
Minimizing the parameters $k_t$, for all $1\le t\le m$, we obtain
\[
    \sum_{t=1}^{m} k_t
    \sim \sum_{t=1}^{(n-1)/2} \left( (t(n-t) + 1/4)^{1/2} + 1/2 \right)
    \sim \int_0^{(n-1)/2} \left( (x(n-x) + 1/4)^{1/2} + 1/2 \right) dx.
\]
Denoting $g(x) := (x(n-x) + 1/4)^{1/2} + 1/2$ and letting $x = \frac{n}{2} - \sqrt{\frac{n^2 + 1}{4}} \cos\theta$, we get 
\begin{align*}
    \int_{0}^{(n-1)/2} g(x) dx
    &= \sqrt{\frac{n^2 + 1}{4}} \int_{\arccos \frac{n}{\sqrt{n^2 + 1}}}^{\arccos \frac{1}{\sqrt{n^2 + 1}}} \left( \frac{n^2 + 1}{4} - \frac{n^2 + 1}{4} \cos^2 \theta \right)^{1/2} \sin\theta d\theta 
    + \frac{n-1}{4} \\
    &= \frac{n^2 + 1}{4} \left( \frac{\pi - 4 \arctan(1/n)}{4} \right) 
    + \frac{n-1}{4} .
\end{align*}
Since $\arctan(x) > x - x^3/3$ for $x > 0$, we conclude that
\[ 
    \int_{0}^{(n-1)/2} g(x) dx
    < \frac{\pi}{16} n^2 + \frac{\pi - 4}{16} + \frac{1 - 2n^2}{12n^3}
    < \frac{\pi}{16} n^2 .
\]

Now 
\[
    k_t  
    = \left\lceil g(t) \right\rceil
    \qquad (1 \le t \le m)
\]
implies that
\begin{equation}\label{eq:odd_error1}
    0 
    \le \sum_{t=1}^{(n-1)/2} k_t  - \sum_{t=1}^{(n-1)/2} g(t)
    \le \frac{n-1}{2}.    
\end{equation}
On the other hand, $g(x)$ is a monotone increasing function for $0\le x\le n/2$, and therefore 
\[
    g(t-1) \cdot 1 
    \le \int_{t-1}^{t} g(x)dx
    \le g(t)\cdot 1
    \qquad (1 \le t \le (n-1)/2).
\]
Hence 
\begin{align}\label{eq:odd_error2}
    0
    \le \sum_{t=1}^{(n-1)/2} g(t) - \int_0^{(n-1)/2} g(x)dx
    &\le \sum_{t=1}^{(n-1)/2} g(t) - \sum_{t=1}^{(n-1)/2} g(t-1) \\
    &= g\left( \frac{n-1}{2} \right) - g(0)
    = \frac{n+1}{2} - 1
    = \frac{n-1}{2}.
    \nonumber
\end{align}
Adding inequalities~\eqref{eq:odd_error1} and~\eqref{eq:odd_error2} yields
\[
    0 
    \le \sum_{t=1}^{(n-1)/2} k_t  - \int_0^{(n-1)/2} g(x)dx
    \le n-1 .
\]
Altogether, we obtain 
for odd $n$:
\[
    \Sort_n
    \ge \inv(\pi_0)
    = \binom{n}{2}-\sum_{t=1}^m k_t
    > \binom{n}{2}-\frac{\pi}{16} n^2 - (n-1)
    > \left( \frac{1}{2}-\frac{\pi}{16} \right) n^2 -\frac{3}{2} n.
\]
This completes the proof of the lower bound for all values of $n$.
\end{proof}

\section{Upper bound on the diameter}

In this section we prove the following. 

\begin{theorem}\label{thm:1}
    For every $n\ge 1$,
    \[
        \D(\Gamma_n) 
        \le \frac{3n^2 - 4n + 1}{8}.
    \]
\end{theorem}

The proof of Theorem~\ref{thm:1} requires several lemmas, which may be of independent interest. 

The following expression for the diameter of $\Gamma_n$ follows immediately from Lemma~\ref{lem:2}.

\begin{corollary}\label{cor:diam} 
For every $n\ge 1$,
    \[
        \D(\Gamma_n)
        = \max_{\pi_1,\pi_2 \in S_n} \min \{\inv(\tau) :\, \tau \in \pi_1 \ZZ_n \pi_2^{-1}\},
    \]
    where $\ZZ_n$ is the subgroup of $S_n$ generated by the $n$-cycle $c_n =(1,2,\dots,n)$.
\end{corollary}

\begin{lemma}\label{obs:4}
For every permutation $\tau\in S_n$ and $0\le k\le n$, 
\[
    \inv(\tau)
    \le \binom{n-k}{2}-k + \sum_{i=1}^k \tau(i).  
\]
\end{lemma}

\begin{proof} 
Fix $0\le k\le n$, and evaluate 
\[
   |\{(i,j) :\, i\le k<j,\, \tau(i)>\tau(j)\}|. 
\]
We can assume, without loss of generality, that $\tau(1)<\tau(2)<\cdots<\tau(k)$. 
Then, for every $1\le i\le k$,
\[
    |\{j :\, k<j,\, \tau(i)>\tau(j)\}|=\tau(i)-i.
\]


Thus  
\[
   |\{(i,j) :\, i\le k<j,\, \tau(i)>\tau(j)\}| 
   = \sum_{i=1}^k (\tau(i)-i)
   = \sum_{i=1}^k \tau(i) - \binom{k+1}{2}.
\]
Note that, obviously, 
\[
    |\{(i,j) :\, i< j \le k,\, \tau(i)>\tau(j)\}| +
    |\{(i,j) :\, k < i < j,\, \tau(i)>\tau(j)\}|
    \le \binom{k}{2} + \binom{n-k}{2}. 
\]
Altogether, 
for every fixed $0\le k\le n$,
\begin{align*}
    \inv(\tau)
    = |\{(i,j) :\, i<j,\, \tau(i)>\tau(j)\}| 
    &\le \sum_{i=1}^k \tau(i) - \binom{k+1}{2} + \binom{k}{2} + \binom{n-k}{2} \\
    &= \sum_{i=1}^k \tau(i) - k + \binom{n-k}{2}.
    \qedhere
\end{align*}
\end{proof}

The key lemma in the proof of Theorem~\ref{thm:1} is the following.  

\begin{lemma}\label{lem:3}
    For 
    any $\pi_1,\pi_2 \in S_n$ and 
    any $0 \le k \le n$, there exists $\tau \in \pi_1 \ZZ_n \pi_2^{-1}$,  such that
    \[
        \sum_{i=1}^k \tau(i)
        \le \frac{k(n+1)}{2}. 
    \]
\end{lemma}

\begin{proof}
First notice that, for any $\pi_1, \pi_2 \in S_n$ and any $1\le i\le n$,
\begin{align*}
    \{\tau(i) :\, \tau \in \pi_1 \ZZ_n \pi_2^{-1}\}
    = \{\pi_1 c_n^j \pi_2^{-1}(i) \,:\, 1\le j \le n\} 
    = \{\pi_1(t) :\, 1\le t \le n\}
    = \{1,\dots,n\}.
\end{align*}
It follows that, for any $1 \le i \le n$, the expected value, in position $i$, of a random permutation in $\pi_1 \ZZ_n \pi_2^{-1}$ is  
\[
    E_{\tau \in \pi_1 \ZZ_n \pi_2^{-1}} \, [\tau(i)]
    = \frac{1}{n}\sum_{j=1}^n \pi_1 c_n^j \pi_2^{-1}(i)
    = \frac{1}{n} (1 + \dots + n) 
    = \frac{n+1}{2}.
\]
Hence, the expected sum of values of $\tau \in \pi_1 \ZZ_n \pi_2^{-1}$ on a prefix of length $k$ is  
\[
    E_{\tau \in  \pi_1 \ZZ_n \pi_2^{-1}} \left[ \sum_{i=1}^k \tau(i) \right]
    = \sum_{i=1}^k E_{\tau \in  \pi_1 \ZZ_n \pi_2^{-1}} \,[\tau(i)]
    = k\cdot \frac{n+1}{2}.
\]    
There must exist $\tau \in \pi_1 \ZZ_n \pi_2^{-1}$ for which $\sum_{i=1}^k \tau(i)$ is smaller or equal to its expected value, and this completes the proof. 
\end{proof}

\begin{proof}[Proof of Theorem~\ref{thm:1}] 
By Lemma~\ref{obs:4} and Lemma~\ref{lem:3}, for any $\pi_1, \pi_2 \in S_n$ and any $0 \le k \le n$,
\[
    \min_{\tau \in \pi_1 \ZZ_n \pi_2^{-1}} \inv(\tau)
    \le \binom{n-k}{2} - k + \min_{\tau \in \pi_1 \ZZ_n \pi_2^{-1}} \sum_{i=1}^{k} \tau(i)
    \le \binom{n-k}{2} - k + \frac{k(n+1)}{2}
    = \binom{n}{2} - \frac{k(n-k)}{2}.
\]
The RHS is minimized for $k = n/2$ (if $n$ is even) or $k = (n \pm 1)/2$ (if $n$ is odd). It follows that
\[
    \min_{\tau \in \pi_1 \ZZ_n \pi_2^{-1}} \inv(\tau)
    \le \binom{n}{2} - \frac{n^2 - \chi(n \text{ odd})}{8}
    = \frac{3n^2 - 4n + \chi(n \text{ odd})}{8},
\]
where $\chi(n \text{ odd})$ is $1$ if $n$ is odd, and $0$ otherwise. 
Corollary~\ref{cor:diam} completes the proof.
\end{proof}


\section{Final remarks and open problems}

A strengthening of Theorem~\ref{thm:main1} is desired. 


\begin{conjecture}\label{conj:0}
    The following asymptotic equality holds:
    \[
        \Sort_n
        = \frac{8-\pi}{16} \cdot n^2 + O(n).
    \]
\end{conjecture}

Recall the permutation $\pi_0$ from~\eqref{eq:defpi}. 
Computer experiments suggest that, furthermore, for every $n\ge 1$
\begin{equation}\label{eq:conj0}
\Sort_n=\inv(\pi_0).
\end{equation}
This was verified for every $n\le 12$. 


\smallskip

Recall that the maximal value of the $\minv$ statistics on $S_n/\ZZ_n$ is equal to the Coxeter sorting time $\Sort_n$ (Corollary~\ref{lem:1}). Hence, 
Theorem~\ref{thm:main1}  provides bounds on this maximal value.
The actual distribution of $\minv$ deserves further study. We conjecture the following. 

\begin{conjecture}
    The generating function of $\minv$ on $S_n/\ZZ_n$ is unimodal.
\end{conjecture}

This conjecture 
was verified for every $n\le 12$.



\smallskip

The $\winv$ statistic, from Definition~\ref{def:winv}, 
plays a crucial role in the proof 
of the upper bound on $\Sort_n$. 
Lemma~\ref{lem:u6} implies the following geometric interpretation of $\winv$:   
\[
    \frac{\winv(\pi)}{\sum_{i=1}^n i^2} 
    = 1 - \cos \alpha_\pi\,, 
\]
where $\alpha_\pi$ is the angle between the vectors $v_\pi :=(\pi(1),\dots,\pi(n))$ and $v_{id} :=(1,2,\dots,n)$ in $\RR^n$. 
The distribution of $\winv$ is mysterious. 
Note that, by Observation~\ref{obs:u2}, the $\winv$ generating function is palindromic. 

\medskip 

The set of adjacent transpositions $\Phi_n := \{(i,i+1):\ 1\le i<n\}$ may be augmented by the transposition $(n,1)$, to yield a cyclically symmetric set:
\[
    \Phi^{cyc}_n
    := \{(i,i+1 \bmod n) :\, 1\le i\le n\}
    = \Phi_n \sqcup \{(n,1)\}.
\]
It is natural to sort cyclic permutations of size $n$ by conjugation with elements of this cyclically symmetric set. 
Let $\Gamma^{cyc}_n$ be the undirected graph whose vertex set is the class $C_n$ of $n$-cycles in $S_n$, and two cycles $c_1, c_2 \in C_n$ are adjacent if there exists a transposition $s \in \Phi^{cyc}_n$ such that $c_2 = s c_1 s$.
Note that $\Gamma^{cyc}_n$ is different from the underlying graph of~\cite{AAR}. 
 Actually, even for $n=4$ these graphs are non-isomorphic. 
However, their sorting times are always equal, as shown below. 


\begin{proposition}
    The sorting time for the graph $\Gamma^{cyc}_n$ 
    is 
    \[
        \Sort(\Gamma^{cyc}_n):= \max_{c \in C_n} \dist_{\Gamma^{cyc}_n} (c,c_n)
        = \left\lfloor \frac{(n-1)^2}{4} \right\rfloor.
    \]
    The diameter of this graph satisfies
    \[
        \left\lfloor \frac{(n-1)^2}{4} \right\rfloor
        \le \D(\Gamma^{cyc}_n)
        \le \left\lfloor \frac{n^2}{4} \right\rfloor .
    \]
\end{proposition}

\begin{proof}
    Recall, from Remark~\ref{rem:iso}, the bijection between $C_n$ and $S_n/\ZZ_n$ defined by mapping a cycle $c =(a_1, \ldots, a_n) \in C_n$ to the coset $\bar{c}\,\ZZ_n \in S_n/\ZZ_n$, where $\bar{c} = [a_1,\ldots,a_n] \in S_n$.
    For any $s \in \Phi^{cyc}_n$, transforming $c$ to $scs$ amounts to interchanging two cyclically-adjacent values in the one-line notation of $\bar{c}$ (where $n$ and $1$ are considered cyclically-adjacent values). This, in turn, is equivalent to swapping two cyclically-adjacent {\em positions} in the one-line notation of $\bar{c}^{-1}$. It follows that the sorting time for $\Gamma^{cyc}_n$ is the same as the sorting time for the graph considered in~\cite{AAR}, which is $\left\lfloor \frac{(n-1)^2}{4} \right\rfloor$. 
    The lower bound on the diameter of $\Gamma^{cyc}_n$ follows.

    In the other direction: by~\cite{ZBSY}, the diameter of $X(S_n,\Phi^{cyc}_n)$, the Cayley graph of the symmetric group $S_n$ with respect to the generating set $\Phi^{cyc}_n$, is 
    $\left\lfloor \frac{n^2}{4} \right\rfloor$. 
    Since 
    $\Gamma^{cyc}_n$ is a quotient of $X(S_n,\Phi^{cyc}_n)$, the upper bound follows. 
\end{proof}




\medskip

A natural generalization of the graph $\Gamma_n$ from Definition~\ref{def:1}  is the following. 
Let $\mu$ be an integer partition of $n$, and let $C_\mu$ be the conjugacy class consisting all permutations in $S_n$ of cycle type $\mu$.  

\begin{definition}\label{def:f1}
    The graph $\Gamma_\mu$ is the undirected graph with vertex set $C_\mu$, in which two permutations $\pi, \sigma \in C_\mu$ are adjacent if $\sigma = s_i \pi s_i$ for some adjacent transposition $s_i \in \Phi_n$.     
\end{definition}


\begin{problem}\label{pb:f1}
    Find the diameter of the graph $\Gamma_\mu$.
\end{problem}



Using properties of the involutive weak order 
~\cite[Definition 1.3]{APR}
and the Bruhat order on fixed-point-free involutions
~\cite{RS, DS, H2, CCT}, 
Avni and Bagno~\cite{Avni-Bagno} proved that 
for  
$\mu=(2^n)$,  
the diameter of $\Gamma_{(2^n)}$ is equal to $(2n-1)!!$. 

\bigskip

\noindent
{\bf Acknowledgments.} 
We thank Noga Alon for useful discussions and suggestions, 
and an anonymous referee for helpful comments.

\end{document}